\documentclass[a4paper,11pt,reqno]{amsart} 

\pdfoutput=1

\usepackage{fullpage}

\usepackage[sort&compress,comma,square,numbers]{natbib}

\usepackage[indent,skip=.2\baselineskip]{parskip}

\usepackage[T1]{fontenc}
\usepackage[utf8]{inputenc}

\usepackage{lmodern}

\usepackage{microtype} 

\usepackage{url}

\usepackage{amsmath}
\usepackage{amssymb}
\usepackage{mathrsfs} 
\usepackage{bm}
\usepackage{bbm} 

\usepackage[dvipsnames,table]{xcolor}

\usepackage{thmtools,mathtools}
\usepackage[unicode,psdextra,pdfusetitle,
bookmarks, bookmarksdepth=2, colorlinks=true, linkcolor=blue!50!black, citecolor=Black!70, urlcolor=blue!50!black,hypertexnames=false]{hyperref}

\usepackage{bookmark}
\hypersetup{bookmarksnumbered}

\usepackage{amsthm}
\usepackage[nameinlink]{cleveref}

\newtheorem{theorem}{Theorem}[section]

\newtheorem{proposition}[theorem]{Proposition}
\newtheorem{corollary}[theorem]{Corollary}

\theoremstyle{definition}
\newtheorem{definition}[theorem]{Definition}
\newtheorem{example}[theorem]{Example}

\newtheorem{remark}[theorem]{Remark}

\newtheorem{theoremalphabetic}{Theorem}

\crefname{ex}{Example}{Examples}
\crefname{thm}{Theorem}{Theorems} 
\crefname{lem}{Lemma}{Lemmas}
\crefname{prop}{Proposition}{Propositions}
\crefname{cor}{Corollary}{Corollaries} 
\crefname{con}{Conjecture}{Conjectures} 
\crefname{def}{Definition}{Definitions}
\crefname{alg}{Algorithm}{Algorithms}
\crefname{rmk}{Remark}{Remarks}
\crefname{thmalph}{Theorem}{Theorems}

\numberwithin{equation}{section}
\numberwithin{figure}{section}

\DeclareTextFontCommand{\bfemph}{\bfseries\em}
\newcommand{\term}{\bfemph}

\usepackage{graphicx}
\usepackage[font=scriptsize,labelfont=bf]{caption}
\usepackage{subcaption}
\usepackage{float}
\usepackage{wrapfig}

\usepackage{multirow}
\usepackage{array}

\usepackage{enumitem}

\definecolor{specialblue}{rgb}{0.07,0.35,0.57}
\definecolor{specialred}{rgb}{0.7,0.2,0.2}

\usepackage[perpage]{footmisc}

\usepackage{pgf,tikz} 
\usetikzlibrary{arrows}
\usetikzlibrary{decorations.markings}
\usetikzlibrary{matrix}
\usetikzlibrary{cd}
\usetikzlibrary{babel}
\usetikzlibrary{positioning}
\usetikzlibrary{shapes}
\usetikzlibrary{arrows.meta,petri}
\usetikzlibrary{calc}

\usepackage{pgfplots}
\pgfplotsset{compat=1.8}

\DeclareMathOperator{\conv}{conv} 
\DeclareMathOperator{\Trop}{Trop}  
\DeclareMathOperator{\argmin}{argmin}  

\newcommand{\N}{{\mathbb{N}}}
\newcommand{\Z}{{\mathbb{Z}}}

\newcommand{\R}{{\mathbb{R}}}

\newcommand{\mC}{{\mathcal{C}}}

 \renewcommand{\k}{\kappa}

\usepackage[version=3]{mhchem}

\title{Persistence of bounded trajectories in positively transverse reaction networks}

\author{Elisenda Feliu}
\date{\today}

\begin{document}


\begin{abstract}
We introduce positively transverse   reaction networks and show that, with mass-action kinetics, no bounded trajectory with positive initial condition has a boundary omega-limit point. These networks are characterized by a geometric property of a collection of polytopes defined by a subset of complexes of the network. 
\end{abstract}

\maketitle


\section{Introduction}
We consider reaction networks with mass-action kinetics and the problem of deciding whether trajectories of the associated mass-action systems admit boundary $\omega$-limit points. Much of the previous work to understand this problem has focused on the resolution of several conjectures on the long-term behavior of the trajectories, which in turn were motivated by the Global Attractor Conjecture (GAC). 
A proof of this conjecture was recently proposed in version 3 of Craciun's preprint \cite{craciun:GAC}, which completes the approach outlined in the 2015 version.

The setting of the GAC dates back to the 1970s and considers a mass-action system that has a positive complex-balancing steady state.  Horn and Jackson showed that in this case every stoichiometric compatibility class intersecting the positive orthant contains a unique positive steady state, which is itself complex-balancing~\cite{HornJackson1972}.  
For every complex-balancing steady state $c^*\in \R^n_{>0}$, they showed the existence of  a strict Lyapunov function on the positive part of the stoichiometric compatibility class containing  $c^*$. In particular, $c^*$ is locally asymptotically stable relative to its compatibility class. The GAC, first formulated by Horn~\cite{Horn74} and later popularized under this name in \cite{TDS}, asserts that this local asymptotic stability is in fact global: every trajectory with positive initial condition converges to the unique complex-balancing steady state in its stoichiometric compatibility class.

Siegel and MacLean  showed in~\cite{SiegelMacLean} that the GAC holds if the trajectories do not have $\omega$-limit points at the boundary of $\R^n_{\geq 0}$, and  efforts to prove the GAC  led to the formulation of broader conjectures on the behavior of the  trajectories near the boundary. These include, for example, the \emph{persistence conjecture}  \cite{CNP},\cite[Rem. 6.1.E]{Feinberg1987}, the \emph{boundedness conjecture} \cite{Anderson2011} and the  (extended) \emph{permanence conjecture} \cite{CNP}. See \cite[\S4]{GMS2} and the introductions of \cite{Pantea,craciun:GAC}  for a more in-depth exposition of these conjectures. 
The persistence, boundedness and permanence conjectures for weakly reversible networks are included in the OpenAI data dump from October 2026. 

The past efforts to prove the GAC enriched the reaction network literature with a better understanding of the behavior of trajectories of mass-action systems close to the boundary, in some cases beyond complex-balancing systems and weakly reversible networks. Early results focused for example on the geometric study of the faces of stoichiometric compatibility classes, viewed as  polyhedra. Craciun, Dickenstein, Shiu, and Sturmfels~\cite{TDS}  showed that vertices of these polyhedra cannot be $\omega$-limit points of positive trajectories of complex-balancing systems (see also Anderson~\cite{Anderson08}, who obtained this for weakly reversible deficiency zero networks).  Subsequent work considered other classes of boundary faces, including  facets (Anderson and Shiu \cite{AndersonShiu10}) and weakly dynamically non-emptiable faces  (Johnston and Siegel~\cite{JohnstonSiegel}). 

Related methods established the GAC for several important families, including networks with a single linkage class (Anderson  \cite{Anderson11}), networks with three species (Craciun, Nazarov and Pantea \cite{CNP}), networks with stoichiometric compatibility classes of dimension  two (Anderson and Shiu \cite{AndersonShiu10}) or   three (Pantea \cite{Pantea}), and \emph{strongly endotactic} reaction networks (Gopalkrishnan, Miller and Shiu \cite{GMS2}). The latter networks are defined by a geometric condition that encodes that  reactions dominating in extreme directions should point back toward the interior of the reactant polytope rather than further outward. 
The PhD thesis of Lawrence \cite{Lawrence-thesis} used the tier approach of \cite{Anderson11} to prove the GAC for a class of reversible, mass-preserving networks. Wiuf showed the GAC for a class of networks with complex-balancing embedded networks \cite{wiuf:GAC}. 

Here we revisit the geometric viewpoint used to show that strongly endotactic networks are  permanent \cite{GMS2}, which  relates limit directions of trajectories to the geometry of the convex hull of (some) complexes of the network.   We focus exclusively on the existence of boundary $\omega$-limit points of bounded trajectories, and hence we do not address the more general problem of when trajectories are   bounded.  This allows us to focus on  a smaller set  of limit directions compared to \cite{GMS2}. 
 This simplification was also made in \cite{Pantea,CNP}, where \emph{lower endotactic} networks were considered.  

 We introduce the class of \emph{positively transverse networks}, which includes, but is not limited to, all networks such that,
for every vector  $a\in \R^{n}_{\geq 0}$ not in the orthogonal complement of the stoichiometric subspace, some complex in the set
\[\Gamma\coloneqq \argmin \{a \cdot y: y \text{ complex}\} \] 
is a reactant of a reaction with product not in $\Gamma$ (see \Cref{def:positive}). 
The precise definition relaxes this requirement to hold only for a subset of such vectors. 

Strongly endotactic networks that are weakly reversible are positively transverse, as they satisfy this condition  for all vectors $a\in \R^n$ not in the orthogonal complement of the stoichiometric subspace (that is, not only for vectors with nonnegative entries). A notable difference from the previous work \cite{GMS2} is that in the definition of $\Gamma$ above, we consider 
\emph{all} complexes of the network, and not only the reactants. This makes the proofs more straightforward at the cost of having a more restrictive condition on the complexes. This distinction has no effect if the network is weakly reversible or \emph{source-only} \cite{Anderson2020Classes}.

Our main result, \Cref{thm:main}, is the following. 

\begin{theoremalphabetic}\label{thm:A}
If a reaction network is positively transverse, then 
no bounded trajectory of any associated mass-action system with initial condition in $\R^n_{>0}$ has an $\omega$-limit point at the boundary of $\R^n_{\geq 0}$. 
\end{theoremalphabetic}

Positively transverse networks need not be weakly reversible, nor even endotactic (see network (c) in \Cref{ex:ex}). Consequently, \Cref{thm:A} does not follow from the recently proposed proof of the GAC by Craciun \cite{craciun:GAC} or the AI-generated proofs of the persistence and permanence conjectures \cite{AI-persistence,AI-permanence}, all of which still await validation by the
mathematical community. Nor would \Cref{thm:A} follow from the extended permanence conjecture \cite{CNP}, which concerns endotactic networks. 

In order to prove  \Cref{thm:A} and derive its consequences, we start with a technical study of logarithmic directions of trajectories with an $\omega$-limit point at the boundary. This is covered in \Cref{sec:technical}.
In \Cref{sec:positive} we use these technical results to derive \Cref{thm:A}. 
A feature of this work is that the proofs are elementary.

\medskip
\noindent
\textbf{Funding. } This project has been funded by the European Union under the Grant Agreement number 101044561, POSALG. Views and opinions expressed are those of the authors only and do not necessarily reflect those of the European Union or European Research Council (ERC).

\medskip
  \noindent
 \textbf{AI disclosure. } The research for this work started without the use of AI in summer 2025.  During this phase the main  ideas were identified: the function $g$, the strategy of proof, and a weaker version of \Cref{thm:main} for weakly reversible networks requiring the faces selected by nonnegative vectors to be transverse. The project was then put on hold for several months due to a lack of time. Upon resuming the work, amid the paradigm shift in mathematics brought about by LLMs, I often engaged in discussions with GPT-5.6 Sol (OpenAI) to address technical aspects of the proofs in \Cref{sec:technical} and to quickly reject ideas. I mainly got help with the proof of \Cref{prop:Wprimepos}, which I  initially only had for trajectories converging to the boundary. The rest of the proofs in that section are based mostly on my own arguments.  GPT-6.1 Sol (OpenAI) and Claude Opus 5.5 (Anthropic) were used to detect inaccuracies and verify the correctness of the arguments at the final stage. No model has been used  to write or to structure the manuscript. The author assumes responsibility for all content.

\medskip
\noindent
\textbf{Notation. } 
We let $\R_{>0}$ and $\R_{\geq 0}$ denote the sets of positive and nonnegative real numbers, respectively. For a set $B\subseteq \R^n$ we let $B^\circ$ denote its relative interior. 
For $a\in \Z^\ell_{\geq 0}$  and $x\in \R^\ell_{\geq 0}$, we let   $x^a = \prod_{i=1}^\ell x_i^{a_i}$ with the convention that $0^0 = 1$.  For $v,w\in \R^n$, $v\cdot w$ refers to their scalar product.

\section{Logarithmic limits and \texorpdfstring{$\omega$}{omega}-limit points}\label{sec:technical}

In this section we develop the technical results towards our main theorems in \Cref{sec:positive}. 

\subsection{Preliminaries} We start by introducing several mathematical objects and definitions. 
We consider a reaction network $G=(V,E)$ with $n$ species numbered  $X_1,\dots,X_n$, where $V\subseteq \Z^{n}_{\geq 0}$ is the set of complexes and $E$ is the set of reactions. 
As usual, a complex $y\in V$ is also written as $y_1X_1+\dots+y_nX_n$ in examples. 

With mass-action kinetics and a fixed choice of reaction rate constants $\k\in \R^{E}_{>0}$,  the associated mass-action system, modeling the time evolution of the concentration $x_i$ of $X_i$  for $i=1,\dots,n$, is 
\begin{equation}\label{eq:mass}
\dot{x} = f(x)\qquad x\in \R^n_{\geq 0}\,,
\end{equation}
where
\begin{equation}\label{eq:f} 
f_i(x)\coloneqq \sum_{y\rightarrow y'\in E} \k_{y\rightarrow y'} (y_i'-y_i) x^y\qquad \text{for } i=1,\dots,n\,. 
\end{equation}
We say that a trajectory $x(t)$ of \eqref{eq:mass} is \term{positive} if the initial condition is in $\R^n_{>0}$. Observe that $\R^n_{>0}$  is forward invariant for the trajectories of \eqref{eq:mass} by \cite{volpert}. If $x(t)$ is defined for all $t\geq 0$, for example if it is bounded, we let $\omega(x(t))$ denote the set of $\omega$-limit points.

The \term{stoichiometric subspace} of the network is $S\coloneqq \langle y'-y : y\rightarrow y'\in E\rangle$. For every $x\in \R^n_{\geq 0}$, the set  $(x+S)\cap \R^n_{\geq 0}$  is called its stoichiometric compatibility class  and contains the trajectory with initial condition $x$. A nonzero vector $w\in S^\perp \cap \R^n_{\geq 0}$ is said to be a \term{nonnegative conservation vector}.

For a point $x\in \R^n$ we refer to the set $\{i: x_i=0\}$ as the \term{zero support} of $x$ and to  the set $\{i: x_i\neq 0\}$ as its \term{support}. 

For a nonempty set of indices $Z\subseteq \{1,\dots,n\}$, we let $G_Z=(V_Z,E_Z)$ be the network obtained by considering all linkage classes of $G$ that are not constant on the species indexed by $Z$ and that involve a species $X_i$ for some $i\in Z$. Specifically,  a linkage class $\mC=(V_\mC,E_\mC)$ is part of $G_Z$ exactly when
\begin{equation}\label{eq:GZ}
y'_i-y_i\neq 0 \quad \text{for some }i\in Z\text{ and }y\rightarrow y'\in E_\mC\,.
\end{equation}
In some works, the projection of $G$ onto the species indexed by $Z$ is considered instead, see e.g. \cite{Anderson11,GMS2}. Here,   the network $G_Z$ depends formally on $n$ species, although some species may not appear in any of the complexes. 

Given $y\in \R^n$, we let 
\[ \bar{y} = (y_i)_{i \in Z} \]
 denote the projection onto the  coordinates indexed by $Z$. 
We let $S_Z$ be the projection of $S$ onto the coordinates $i\in Z$, which by construction equals
\[ S_Z  = \big\langle (y_i'-y_i)_{i\in Z} :  y\rightarrow y'\in E \big\rangle = \big\langle \bar{y}'-\bar{y} :  y\rightarrow y'\in E_Z \big\rangle  \subseteq \R^{Z}\,.\]

We let  $P_Z$ be the convex hull of $V_Z\subseteq \R^n$. 
We denote the $Z$-inner normal cone at a face $\Gamma$ of $P_Z$ as 
\[ N_Z(\Gamma) \coloneqq \{a \in \R^Z :  \Gamma\subseteq \argmin  \{  a\cdot \bar{y} : y \in P_Z \}\} \, , \]
and let 
\[ N_Z^*(\Gamma) \coloneqq \{a \in \R^Z :  \Gamma =  \argmin  \{ a\cdot \bar{y} : y \in P_Z \}\} \, . \]
We say that $a\in \R^{Z}$ determines the face $\Gamma$  if $a\in N_Z^*(\Gamma)$.

Let  $G'=(V',E')$ be a subnetwork of $G_Z$ consisting of a union of linkage classes and let $P'$ be the convex hull of $V'$. 
We say  that a face $\Gamma$ of $P'$ is \term{$Z$-transverse} for $G'$ if  there exists a reaction $y\rightarrow y'\in E'$ such that  $y\in \Gamma$ and $y'\notin \Gamma$ (this choice of wording adapts \cite[Def. 4.3]{anderson:cappelletti:tier}).

We consider the standard simplex in $\R^Z$ defined as
\begin{equation}\label{eq:delta}
\Delta_Z \coloneqq  \Big\{ a\in \R^Z_{\geq 0} :  \sum_{i \in Z} a_i =1\Big\} \subseteq \R^Z_{\geq 0} \,. 
\end{equation}
 Its relative interior $\Delta^\circ_Z$ consists of the subset of points with all coordinates positive. 
\smallskip
Given $x\in \R^n_{>0}$ such that $u\coloneqq \prod_{i\in Z} x_i \neq 1$, we  let 
\begin{equation}\label{eq:AR}
a(x) \coloneqq  \left(\frac{\log x_i }{\sum_{k\in Z}  \log x_k}\right)_{i\in Z}\,.
\end{equation}
We will repeatedly use that 
\begin{equation}\label{eq:xR} 
x_i =  u^{a_{i}(x)} \qquad \text{  for all }i\in Z\,,
\end{equation}
since
\[ 
a_{i}(x)   \log u = \frac{\log x_i}{\sum_{k\in Z}  \log x_k} \log\Big(\prod_{k\in Z} x_k \Big) =\log x_i  \, .
\]

With $f_i$ as in \eqref{eq:f}, we consider the function
\begin{equation}\label{eq:ga}
\begin{aligned}
g(x) &\coloneqq \sum_{i\in Z} a_i(x) f_i(x)  =   \sum_{i\in Z} a_{i}(x)  \sum_{y\rightarrow y'\in E_Z}  (y'_i-y_i) \k_{y\rightarrow y'}  x^y   \\
&
= \sum_{y\rightarrow y'\in E_Z}    (a(x) \cdot (\bar{y}' -   \bar{y})) \k_{y\rightarrow y'}  x^{y}\,.  
\end{aligned}
\end{equation}
Here we have used that a reaction contributes a summand to  $f_i$ only if $y_i'-y_i\neq 0$. 
When considering the function $g$, the set $Z$ is implicit from the context. 

More generally, if $G'=(V',E')$ is a subnetwork of $G_Z$, we consider 
\[ g_{G'}(x) \coloneqq    \sum_{y\rightarrow y'\in E'}    (a(x)\cdot (\bar{y}' -   \bar{y})) \k_{y\rightarrow y'}    x^{y}\,.  \]
The function $g$ in \eqref{eq:ga} is the sum of $g_\mC$ over all linkage classes $\mC$ of $G_Z$.

 \subsection{Negative logarithmic limits} 
 
\begin{definition}\label{def:bad}
Let $x(t)$ be a  positive trajectory of \eqref{eq:mass},  $Z\subseteq \{1,\dots,n\}$ nonempty, and   $h\colon D\rightarrow \R$ be a   function defined on $D\coloneqq \{x\in \R^n_{> 0} : \prod_{i\in Z} x_i \neq 1\}$. A point $a_\infty\in \Delta_Z$ is called a \term{negative logarithmic limit} for $(x(t),Z,h)$ if there 
  exists  a sequence $\{t_j\}_{j\in \N}\subset \R_{>0}$ satisfying
\begin{itemize}
\item $\{t_j\}$ tends to $\infty$ as $j$ tends to $\infty$.
\item $x_i(t_j)<1$ for all $i\in Z$ and $j\in \N$.
\item  $\{x(t_j)\}$ converges to a point in $\R^n_{\geq 0}$ with zero support $Z$.
\item $\{a(x(t_j))\}$ converges to $a_\infty$.
\item $h(x(t_j)) <0$ for all $j\in \N$.
\end{itemize}
\end{definition}

\begin{proposition}\label{prop:Wprimepos}
Let $x(t)$ be a bounded positive trajectory of \eqref{eq:mass}, $Z\subseteq \{1,\dots,n\}$ be nonempty, and consider the associated function $g$ from \eqref{eq:ga}. 

If $Z$ is inclusion-maximal among all possible zero supports of points in $\omega(x(t))$, then there exists a negative logarithmic limit for $(x(t),Z,g)$.  
\end{proposition}
\begin{proof}
Let $x^*$ be  an $\omega$-limit point of the trajectory with zero support $Z$ (which exists by hypothesis). 
Consider the following function on $\R^n_{>0}$:
\[ W(x)\coloneqq  \sum_{i\in Z}\, (x_i\log x_i - x_i)  \,.\]
Observe that $W$ can be extended continuously to $\R^n_{\geq 0}$. 
Using the definition of $g(x)$ and that $\dot{x}_i(t)=f_i(x(t))$, provided that $\prod_{k\in Z} x_k(t) \neq 1$ it holds that
\begin{equation}\label{eq:Wg}
\tfrac{d}{dt} W(x(t)) =     \sum_{i\in Z}  (\log x_i(t))  \, f_i(x(t))  =  \Big(\sum_{k\in Z}  \log x_k(t) \Big) g(x(t))  \, . 
\end{equation}
If a sequence $x(t_j)$ converges to $x^*$, then $ \sum_{k\in Z}  \log x_k(t_j)<0$ for $j$ sufficiently large. Using \eqref{eq:Wg} and \Cref{def:bad}, we  aim to find a sequence of points in the trajectory at which $\tfrac{d}{dt} W$ is positive. This will be obtained by constructing two suitable sequences  $\{x(d_j)\}$ and $\{x(s_j)\}$ such that $d_j<s_j$ and $W(x(d_j)) < W(x(s_j))$ for all $j\in \N$. 

To this end,  choose a sequence $\{\delta_j\}_{j\in \N}$ converging to zero such that 
\[0<\delta_j<\min \{x_i(0), 1 :  i\in Z\}\,. \] 
Consider the box
\[B_j \coloneqq \{x\in \R^n_{>0} : 0<x_i\leq \delta_j,\ i\in Z\}\, \]
and let 
\[ \eta_j \coloneqq    \delta_j  \log \delta_j -\delta_j  <0\, . \]

For a fixed $j$, if $x\in B_j$, then  $x_i \log x_i  - x_i  < 0$ for all $i\in Z$ as $x_i<1$. 
If   additionally   $x_i=\delta_j$ for some $i\in Z$, then 
\begin{equation}\label{eq:WRbound}
W(x) = \delta_j  \log \delta_j -\delta_j   + \sum_{k\in Z\setminus \{i\}}  (x_k \log x_k  - x_k)  \leq   \eta_j <0\,.
\end{equation}

Choose a sequence $\{x(r_j)\}$ converging to $x^*$ with $\{r_j\}$ converging to $\infty$. As $x_i(r_j)$ converges to zero if $i\in Z$ and $\{W(x(r_j))\}$ also converges to zero, 
there exists a subsequence $\{x(s_j)\}$ satisfying
 \begin{equation}\label{eq:choosetj}
 x(s_j)\in B_j^\circ\,, \qquad \text{ and } \qquad W(x(s_j))> \frac{\eta_j}{2}\,. 
\end{equation}

As $x(0) \notin B_j$ by construction and $x(s_j)\in B_j^\circ$, there exists a last time $d_j<s_j$ at which $x(s)$ enters the box $B_j$. Specifically,   for $i\in Z$ let
\[ d_{j,i} \coloneqq  \max\{s\in [0,s_j] : x_i(s)= \delta_j \} \quad \text{and}\quad d_{j} \coloneqq  \max_{i\in Z} \{d_{j,i}\}\,.\]
This construction gives that  
\[ d_j<s_j \,, \quad  x(s)\in B_j^\circ\quad \text{ for all }s\in (d_j,s_j] \,,\quad\text{and}\quad \max_{i\in Z} \, x_i(d_j)=\delta_j\,.\] 
In particular  \eqref{eq:WRbound} gives that $W(x(d_j))\leq \eta_j$. Together with 
 \eqref{eq:choosetj} and using that $\eta_j<0$ we obtain that
\[  W(x(d_j))\leq \eta_j < \frac{\eta_j}{2} <    W(x(s_j))\,. \]
The mean value theorem now gives the existence of   $t_j\in(d_j,s_j)$ such that
\[\tfrac{d}{dt} W(x(t_j))>0\, \qquad j\in \N\,. \]
Since $x(t_j)\in  B_j^\circ$, $\{x_i(t_j)\}$ converges to zero for all $i\in Z$, and \eqref{eq:Wg} gives that $g(x(t_j))<0$. 

Additionally, $\{t_j\}$ tends to $\infty$ because $\{d_j\}$ tends to $\infty$. To see this, let $M>0$. As the trajectory remains in $\R^n_{>0}$ by forward invariance, continuity of the trajectory and compactness give that
\[ \alpha_M \coloneqq \min \{ x_i(t) : i\in Z , 0\leq t\leq M \} >0\, . \]
As $\delta_j$ converges to zero, there exists $j_0$ such that $\delta_j<\alpha_M$ for all $j>j_0$. As $x(d_j)\in B_j$, this implies that $d_j>M$ for all $j>j_0$. Hence $\{d_j\}$ tends to $\infty$.

As the trajectory is bounded,  $\{(x_i(t_{j}))_{i\notin Z}\}$ has a convergent subsequence. 
By passing to a subsequence if necessary, $x(t_j)$ converges to a point $x^{**}$, which  in particular is an $\omega$-limit point. We have by construction that $x_i^{**}=0$ for $i\in Z$, and by the maximality of $Z$,  $x_i^{**}\neq 0$ if $i\notin Z$. 
Finally, as by construction $a(x(t_j))\in \Delta_Z$, which is compact, there exists a subsequence of $\{x(t_j)\}$ such that 
$\{a(x(t_j))\}$ converges to some $a_\infty\in \Delta_Z$. 
This gives the desired negative logarithmic limit $a_\infty$. 
\end{proof}

\Cref{prop:Wprimepos} tells us that negative logarithmic limits exist if a bounded trajectory has a boundary $\omega$-limit point, as there will always be an inclusion-maximal zero support of $\omega$-limit points. The strategy to preclude   boundary $\omega$-limit points is to understand properties of negative logarithmic limits and identify network characteristics under which these properties cannot arise.

\begin{proposition}\label{prop:selectpos}
Consider a reaction network $G=(V,E)$ with $n$ species, $Z\subseteq \{1,\dots,n\}$ nonempty, and  let $x(t)$ be a  positive trajectory of \eqref{eq:mass}. Let $G'=(V',E')$ be a nonempty subnetwork of $G_Z$ consisting of a union of linkage classes and let $P'$ be the convex hull of the complexes in $V'$. 

If $a_\infty\in \Delta_Z$ is a negative logarithmic limit for $(x(t),Z,g_{G'})$, then 
there exists a face $\Gamma$ of $P'$ such that 
\begin{itemize}
\item $\Gamma$ is not  $Z$-transverse for $G'$.
\item $\Gamma=\argmin  \{ a' \cdot \bar{y} : y \in P'\}$ for some $a'\in  \Delta^\circ_Z$.
\item $\Gamma$  is a face of the face $\argmin  \{ a_\infty \cdot \bar{y}: y \in P'\}$ of $P'$. 
 \end{itemize}
\end{proposition}
\begin{proof}
Let $\{t_j\}$ be a sequence witnessing \Cref{def:bad} for $a_\infty$. 
Let $u_j\coloneqq \prod_{i\in Z}  x_i(t_j) >0$ and observe that $\{u_j\}$ converges to zero because $\{x_i(t_j)\}$ converges to zero for all $i\in Z$. Also, we have that   $a_j\coloneqq a(x(t_j))\in \Delta^\circ_Z$ and $x_i(t_j) =   u_j^{a_{ji}}$ for $i\in Z$ by \eqref{eq:xR}.
Let 
\[ \Gamma_0 \coloneqq  \argmin  \{ a_\infty \cdot \bar{y} : y \in P'\} \,. \]

For each $j\in \N$, let  $\beta_{j,y}  \coloneqq  \prod_{i\notin Z} x_{i}(t_j)^{y_i}$, consider the sum $\widetilde{u}_j\coloneqq  \sum_{y\in V'}u_j^{a_j\cdot \bar{y}}$ and let 
\begin{align*}
 \tilde{g}_j &\coloneqq  \frac{g_{G'}(x(t_j))}{\widetilde{u}_j}
 =   \sum_{y\rightarrow y'\in E'}    (a_j\cdot (\bar{y}' -   \bar{y})) \k_{y\rightarrow y'} \beta_{j,y}  \frac{u_j^{a_j\cdot \bar{y}}}{\widetilde{u}_j}\,.
 \end{align*} 
 The signs of  $\tilde{g}_j$ and of $g_{G'}(x(t_j)) $ agree and  by the choice of $t_j$ it follows that $\tilde{g}_j<0$  (cf. \Cref{def:bad}). 
 
Consider the vector of quotients 
\begin{equation}\label{eq:z}
z_j\coloneqq  \left(\frac{u_j^{a_j\cdot \bar{y}}}{\widetilde{u}_j}\right)_{y\in V'}\,,
\end{equation}
which we   view in  $\R^{V'}$. Then
\[ \tilde{g}_j =  \sum_{y\in V'}  \Big(\sum_{y\rightarrow y'\in E'}    (a_j\cdot (\bar{y}' -   \bar{y})) \k_{y\rightarrow y'} \Big) \beta_{j,y} \,z_{j,y} \,. \]
Each $z_j$ belongs to the relative interior of the standard simplex in $\R^{V'}$, and hence by compactness of the simplex, by passing to a subsequence, we can assume that $\{z_j\}$ converges to some $z_\infty$ in the simplex. Let 
\[ \Gamma_{V'} \coloneqq \{y\in V' : z_{\infty,y}>0\}\]
 and let us inspect the term $z_{\infty,y_0}$ for $y_0\in V'$ more closely. We have that 
\begin{equation}\label{eq:zy}
 z_{j,y_0}= \frac{u_j^{a_j\cdot \bar{y}_0}}{\sum_{y\in V'} u_j^{a_j\cdot \bar{y}}} 
=  \frac{1}{ 1+ \sum_{y\in V', y\neq y_0} u_j^{a_j\cdot (\bar{y}-\bar{y}_0)}}\, . 
\end{equation}

If $y_0\notin \Gamma_0$, then by definition of $\Gamma_0$ and $P'$,  $\eta\coloneqq a_\infty\cdot  (\bar{y}-\bar{y}_0)<0$ for some $y\in \Gamma_0\cap V'$. 
Since $a_j$ converges to $a_\infty$ (by \Cref{def:bad}) we have
\[ a_j\cdot  (\bar{y}-\bar{y}_0)< \eta/2<0 \quad \text{for $j$ sufficiently large, and hence } \quad
u_j^{a_j\cdot (\bar{y} -   \bar{y}_0)} \geq u_j^{\eta/2}\,.\]
We conclude from this that
$u_j^{a_j\cdot (\bar{y} -   \bar{y}_0)}$
 tends to infinity as $u_j$ tends to zero. This implies that there is at least  one term in the denominator of \eqref{eq:zy}
 that tends to infinity. As all terms in the denominator are positive, it follows that $z_{\infty,y_0}=0$ and hence  $y_0\notin \Gamma_{V'}$. In particular
\begin{equation}\label{eq:face_inclusion}
\Gamma_{V'}\subseteq \Gamma_0 \,.
\end{equation}

Observe  now that 
\begin{equation}\label{eq:au}
a_j\cdot(\bar y'-\bar y) = \frac{-1}{\log u_j} \log\frac{z_{j,y}}{z_{j,y'}}\,,
\end{equation}
for all $y,y'\in V'$, and hence we have
\begin{equation}\label{eq:gj_rewrite}
(- \log u_j)\, \tilde{g}_j =\sum_{y\in V'}   \sum_{y\rightarrow y'\in E'}  
 \k_{y\rightarrow y'} \beta_{j,y} \, \Big(z_{j,y}  \log\frac{z_{j,y}}{z_{j,y'}}\Big)\,. 
 \end{equation}
Note that this expression has the same sign as $\tilde{g}_j$  and that   the coefficients $ \k_{y\rightarrow y'} \beta_{j,y}$ are bounded away from $0$ and $\infty$ (as every $x_i(t_j)$, for $i\notin Z$, converges to a positive number by \Cref{def:bad}). 
 
For a reaction $y\rightarrow y'\in E'$, we look at the term it contributes  in \eqref{eq:gj_rewrite}: 
 \begin{itemize}
 \item 
 If $y\in \Gamma_{V'}$ and $y'\notin\Gamma_{V'}$, 
 then
\[ z_{j,y}\ \text{ tends to } z_{\infty,y}>0 \qquad \text{and}\qquad z_{j,y'}\ \text{ tends to } 0\, ,\]
so
\[ z_{j,y}\log\frac{z_{j,y}}{z_{j,y'}} \text{ tends to } \infty\, .\]
Thus this term tends to $\infty$.
\item If $y,y'\in\Gamma_{V'}$, then  $z_{j,y}\log\frac{z_{j,y}}{z_{j,y'}}$ has a finite limit and hence the sequence stays bounded.

\item 
If $y\notin\Gamma_{V'}$, then  $z_{j,y}>0$ converges to zero and in particular so does $z_{j,y} \log z_{j,y} $. For every $\varepsilon>0$ there exists $j_0$ such that $-\varepsilon<z_{j,y} \log z_{j,y} <0$ for all $j>j_0$. 
Hence, as $z_{j,y'}\leq 1$, we have 
\[ -\varepsilon < z_{j,y} \log z_{j,y}  \leq z_{j,y}\log\frac{z_{j,y}}{z_{j,y'}} \,,\] 
and  the summand is bounded from below  for all sufficiently large $j$. 
\end{itemize}

We conclude that the existence of a reaction from a complex in $\Gamma_{V'}$ to a complex not in $\Gamma_{V'}$ forces $- \log(u_j)\, \tilde{g}_j$ in  \eqref{eq:gj_rewrite}   to tend to $\infty$ and hence  become positive for all $j$ large enough. Therefore the fact that $\tilde{g}_j<0$  for all $j$ sufficiently large implies that there is no such reaction.

\medskip
All that is left is to show that 
\[ \Gamma_{V'}  = \argmin \{ a'\cdot \bar{y}: y\in V'\}\]
for some $a'\in \Delta^\circ_Z$. From this we will obtain that $\Gamma\coloneqq \conv (\Gamma_{V'})$  is the face of $P'$ in the statement: it is not  $Z$-transverse since $\Gamma\cap V'=\Gamma_{V'}$ and \eqref{eq:face_inclusion} gives that $\Gamma$ is a face of $\Gamma_0$.

 To this end, choose $y_0\in\Gamma_{V'}$ and consider the vector subspace
\[ U \coloneqq   \langle\bar{y}-\bar{y}_0 : y\in \Gamma_{V'} \rangle  \subseteq \R^Z \,.\]
Let
\[ w_j \coloneqq  -(\log u_j) a_j  \in \R^Z_{> 0}\,\]
so that by \eqref{eq:au} we have
\[ w_j \cdot   (\bar y-\bar y_0)= \log\frac{z_{j,y_0}}{z_{j,y}} \, . \] 
Using the earlier arguments, we observe that: 
\begin{enumerate}[label=(\roman*)]
\item If $\mu \in U$,    then
$w_j \cdot \mu$ remains bounded as $j$ tends to $\infty$, as $w_j \cdot   (\bar y-\bar y_0)$ remains   bounded when $y\in \Gamma_{V'}$. 
\item For all $i\in Z$, $w_{ji} =-\log  x_i(t_j)$ tends to $\infty$ as $x_i(t_j)$ converges to zero.  
\item $w_j \cdot   (\bar y-\bar y_0)$  tends to $\infty$ if $y\in V'\setminus \Gamma_{V'}$. 
 \end{enumerate}

Choose  an orthonormal basis $\{\mu_1,\dots,\mu_{s}\}$ of $U$ and  consider   
\[\mu(j)= (w_j\cdot \mu_1) \mu_1+ \dots + (w_j\cdot \mu_s) \mu_s \in U \,, \quad \text{so that } w_j = \mu(j)+\nu(j)\ \text{ with }\nu(j)\in U^\perp  \, . \]
For $y\in \Gamma_{V'}$, as $\bar y-\bar y_0 \in U$, we have that $\nu(j)\cdot  (\bar y-\bar y_0) =0$ for all $j$.

By point (i), $w_j\cdot \mu_l$ remains bounded for all $l=1,\dots,s$ and hence  each entry of $\mu(j)$ remains bounded.
By point (ii) it then follows that all entries of $\nu(j)$  tend to $\infty$. 
 In particular, there exists $j_0$ such that  $\nu(j)\in \R^Z_{>0}$ for all $j>j_0$.

Similarly, for $y\in V'\setminus \Gamma_{V'}$, as $\mu(j)\cdot  (\bar y-\bar y_0)$ remains bounded, point (iii) implies that 
$\nu(j) \cdot   (\bar y-\bar y_0)$ tends to $\infty$. 
Using now that  the number of complexes in $V'$ is finite, there exists $j_1>j_0$ such that  $\nu(j) \cdot (\bar y-\bar y_0)>0$ for all $y\in V'\setminus \Gamma_{V'}$ and $j>j_1$. 

For any such $j$, by defining 
\[ a'\coloneqq \frac{\nu(j)}{\sum_{i\in Z} \nu(j)_i} \in\Delta_Z^\circ \quad \text{we have}\quad  
 a'\cdot\bar y \  \begin{cases}  =  a'\cdot\bar y_0 & \text{if } y\in\Gamma_{V'}  \\
>a'\cdot\bar y_0  & \text{if }  y\in V'\setminus \Gamma_{V'} \,. 
\end{cases}
\]
As $\Gamma$ is the convex hull of $\Gamma_{V'}$,  $a'\in \Delta_Z^\circ$ determines $\Gamma$ as desired.
\end{proof}

\begin{remark}
A key point in the proof of \Cref{prop:selectpos} is the normalization of the vector of monomials giving rise to $z_j$ in \eqref{eq:z}. With the notation in the proof, the fact that the convex hull of $\Gamma_{V'}$ is a face of $P'$ could have been directly derived from an argument of positive toric geometry, but this approach did not automatically give that the face was determined by a vector with positive entries. An analogous normalization is used in the proof by  \cite{BorosHofbauer2020} of the permanence conjecture for weakly reversible networks with one linkage class. The writing of $g$ in terms of  $z_{j,y}\log\frac{z_{j,y}}{z_{j,y'}}$, which allows the study of convergence, appears in previous works such as \cite[Lemma 3.2]{manoj:cycles} and  \cite{anderson:cappelletti:tier}. 
\end{remark}

By taking $G'=G_Z$, \Cref{prop:selectpos} tells us a geometric property that negative logarithmic limits must satisfy. 
We now derive an additional property, which builds on the positive tropicalization of a vector subspace, defined as follows (see \cite{SpeyerWilliams2005,ArdilaKlivansWilliams2006,RoseTelek24}). 
Let $U\subseteq \R^\ell$ be a vector subspace. 
A pair of disjoint index subsets $C^+,C^-\subseteq \{1,\dots,\ell\}$ is said to be a \term{signed circuit} of $U^\perp$ if there exists a nonzero vector $v\in U^\perp$   whose support is inclusion-minimal among nonzero vectors of $U^\perp$, and such that $C^+ = \{i : v_i>0\}$ and  $C^- = \{i : v_i<0\}$.

The \term{positive tropicalization} of $U$ is then
\begin{equation}\label{eq:Bergman}
\begin{aligned}
  \Trop^+(U) & \coloneqq \{ a \in \mathbb{R}^{\ell} : 
 \text{for every signed circuit } (C^+,C^-)  \text{ of $U^\perp$ it holds that} \\ & \hspace{3cm} \textstyle\min_{k \in C^+\cup C^-} \, a_k = a_i = a_j   \ \text{ for some }i \in C^+ ,\  j \, \in C^- \}\,.
 \end{aligned}
\end{equation}
Here we use the min convention for the tropicalization. We note that if $U=\R^\ell$, then $\Trop^+(U)=\R^\ell$.

\begin{proposition}\label{prop:bergman}
Consider a reaction network $G=(V,E)$ with $n$ species, $Z\subseteq \{1,\dots,n\}$ nonempty, and  let $x(t)$ be a  positive trajectory of \eqref{eq:mass}. If $a_\infty\in \Delta_Z$ is a negative logarithmic limit for $(x(t),Z,g)$, then 
\[ a_\infty \in  \Trop^+(S_Z) \,.\]
\end{proposition}
\begin{proof}
Let $\{t_j\}$ be  a sequence witnessing \Cref{def:bad} for $a_\infty$. Let $x^*\in \R^n_{\geq 0}$ be the limit of $\{x(t_j)\}$, which by hypothesis has zero support $Z$. All points $x(t_j)$ belong to the same stoichiometric compatibility class, which is a closed set, and hence so does $x^*$. It follows that $x(t_j)-x^* \in S$ for all $j$. Projecting onto the coordinates in $Z$, we obtain that 
\[ \overline{x(t_j)} \in S_Z \subseteq \R^Z\, . \]

Let $u_j\coloneqq \prod_{i\in Z}  x_i(t_j)<1$, which converges to zero because $x_i(t_j)$ converges to zero. By construction $a_j \coloneqq a(x(t_j))\in \Delta_Z^\circ$ and using \eqref{eq:xR} we have  $x_i(t_j) =   u_j^{a_{ji}}$ for $i\in Z$. 

Let $v\in S_Z^\perp$ be nonzero and have inclusion-minimal support among all nonzero vectors of $S_Z^\perp$. 
We now proceed similarly to the proof of \Cref{prop:selectpos}. 
As $0 = v\cdot \overline{x(t_j)}$, after dividing by $\widetilde{u}_j\coloneqq \sum_{k\in Z: v_k\neq 0} u_j^{a_{jk}}>0$, we obtain
\begin{align*}
0 & =  \sum_{i\in Z}  v_i\,  \frac{u_j^{a_{ji}}}{\widetilde{u}_j} \,\quad   \text{ for all $j$}\,.  
\end{align*}
For every $j$, the vector
$z_j\in \R^Z$ such that $z_{ji}=0$ if $v_i=0$, and $u_j^{a_{ji}}/\widetilde{u}_j$ otherwise, 
 belongs to $\Delta_Z$ and satisfies $0=v\cdot z_j$. 
As $\Delta_Z$ is compact, by passing to a subsequence if necessary, $\{z_j\}$ has a limit $z_\infty \in \Delta_Z$ for $j$ tending to $\infty$, which also satisfies 
$0 =    v \cdot  z_{\infty}.$ 

Let $m_{\infty} \coloneqq \min \{a_{\infty,i} : i\in Z, v_i\neq 0\}$. If $a_{\infty,i}>m_{\infty}$ and $v_i\neq 0$ for $i\in Z$, then 
\[ \frac{u_j^{a_{ji}}}{\widetilde{u}_j}  = \frac{1}{ 1+ \sum_{k\neq i, v_k\neq 0} u_j^{a_{jk}-a_{ji}}} \] 
tends to zero: by choosing $k$ such that $a_{\infty, k}=m_{\infty}$ and $v_k\neq 0$, we have that  $a_{\infty,k}-a_{\infty, i}<0$, which implies that $u_j^{a_{jk}-a_{ji}}$ tends to $\infty$ and hence $z_{\infty,i}=0$. 

Hence, $\{i\in Z : z_{\infty,i}>0\} \subseteq \{i\in Z :a_{\infty,i}=m_{\infty},v_i\neq 0\}$. In particular, 
\[ 0 =  \sum_{i\in Z:\, a_{\infty,i}=m_{\infty}} v_i \, z_{\infty,i}\,. \]
As $z_{\infty}$ is nonnegative and nonzero, it follows that there must exist indices $i,k\in Z$ such that  $z_{\infty,i}, z_{\infty,k}\neq 0$ and $v_i>0,v_k<0$. 
In particular 
$m_{\infty}= \min_{s : v_s\neq 0} \, a_{\infty,s} = a_{\infty,i}  = a_{\infty,k} $ with $i \in \{s: v_s>0\}$ and $k \, \in  \{s: v_s<0\}$. 
Since this holds for all nonzero vectors $v$ with minimal support, this holds for all signed circuits and hence $a_\infty\in \Trop^+(S_Z)$.
\end{proof}

\begin{remark}
\Cref{prop:bergman} admits an alternative, arguably more direct, proof using the description of $\Trop^+(U)$ as a limit of logarithms of the set $U\cap \R^n_{>0}$, see \cite[Cor. 4.6]{alessandrini}, \cite[Prop. 4.3]{RoseTelek24}. We choose the approach using the signed circuits due to its geometric character and suitability for computations.
\end{remark}

We obtain the following local statement on logarithmic limits of trajectories that approach the boundary of $\R^n_{\geq 0}$. 

\begin{corollary}\label{cor:gooda}  Consider a reaction network $G$ with $n$ species. Let $Z\subseteq \{1,\dots,n\}$ be nonempty, $g$  the associated function from \eqref{eq:ga}, and $x(t)$ a  positive trajectory of \eqref{eq:mass}. 
 
If $a_\infty\in \Delta_Z$ is a negative logarithmic limit for $(x(t),Z,g)$, then  $G_Z\neq \varnothing$ and there exists  a face $\Gamma$ of $P_Z$ that is not  $Z$-transverse for $G_Z$, and such that 
\[ a_\infty \in   N_Z(\Gamma)  \cap \Trop^+(S_Z)\, \quad\text{and}\quad N_Z^*(\Gamma)\cap \R^{Z}_{> 0} \neq \varnothing \,.\]
\end{corollary}
\begin{proof} 
By \Cref{prop:bergman}, $a_\infty \in  \Trop^+(S_Z)$.  
If $G_Z=\varnothing$, then $g$ is identically zero and $a_\infty$ cannot be a negative logarithmic limit. 
Now apply \Cref{prop:selectpos} to $G'=G_Z$ to obtain the face $\Gamma$  that is not  $Z$-transverse for $G_Z$ and that is determined by a positive vector, hence $N_Z^*(\Gamma)\cap \R^{Z}_{> 0} \neq \varnothing$. As $\Gamma$ is a face of $\argmin \{a_\infty\cdot \bar y : y \in P_Z\}$ by \Cref{prop:selectpos}, 
it follows  that $a_\infty\in N_Z(\Gamma)$. 
\end{proof}

\section{Positively transverse networks}\label{sec:positive}

The results of \Cref{sec:technical} allow us to identify a class of networks for which $\omega$-limit sets of bounded trajectories cannot have boundary points, namely those for which negative logarithmic limits cannot exist.

A nonempty subset $Z\subseteq\{1,\dots,n\}$ is said to be a \term{siphon} of a reaction network $G=(V,E)$ if for all $i\in Z$ and $y\rightarrow y'\in E$ such that $y_i'\neq 0$, there exists $j\in Z$ such that $y_j\neq 0$ (see e.g. \cite{AngeliDeLeenheerSontag2007,ShiuSturmfels}). In  words, if $X_i$ for $i\in Z$ is in the product of a reaction, then the reactant of the reaction includes some species with index in $Z$ as well. 

A  well-known property (see e.g \cite{AngeliDeLeenheerSontag2007,Gopalkrishnan2011Catalysis,DeshpandeGopalkrishnan2014,MarcondesFeliuWiuf2017}), which we exploit below, is that if  the zero support of a point $x^*\in \R^n_{\geq 0}$  contains the support of a nonnegative conservation vector $w$, then $(x^*+S)\cap \R^n_{>0}=\varnothing$. 
 This is easy to see, as any point $z\coloneqq x^*+v \in x^*+S$ satisfies $w \cdot (x^* + v) = w \cdot x^* = 0$. If $z\in \R^n_{\geq 0}$, then $w \cdot z=0$ 
implies $z_i=0$ whenever $w_i\neq 0$. Hence $x^*+S$ does not intersect $\R^n_{>0}$. 
This fact  has been used to identify zero supports of points in stoichiometric compatibility classes intersecting $\R^n_{>0}$, see e.g. \cite{ShiuSturmfels,AngeliDeLeenheerSontag2007,JohnstonSiegel}.

\begin{definition}\label{def:siphon}
Consider a reaction network with $n$ species and a siphon $Z\subseteq \{1,\dots,n\}$. Then $Z$  is said to be:
\begin{itemize}
\item  A \term{relevant siphon} if  $Z$ does not contain the support of a   nonnegative conservation vector.
\item A \term{positively transverse siphon}  if for every face $\Gamma$ of $P_Z$, at least one of the following holds:
\begin{itemize}
\item $N_Z^*(\Gamma)\cap \R^{Z}_{> 0} = \varnothing$.
\item $\Gamma$ is  $Z$-transverse for $G_Z$.
\item $N_Z(\Gamma)\cap \Trop^+(S_Z) \cap \Delta_Z = \varnothing$.
\end{itemize}
\end{itemize}
\end{definition}

\begin{definition}\label{def:positive}
A reaction network is said to be \term{positively transverse} if  every relevant siphon is positively transverse. 
\end{definition}

\begin{proposition}\label{prop:main}
Consider a reaction network and let $x(t)$ be a bounded positive trajectory  of \eqref{eq:mass}.  If $x(t)$ has a  boundary $\omega$-limit point whose zero support $Z$ is a positively transverse siphon, then there exists $Z\subsetneq Z'$ such that $Z'$ is the zero support of an $\omega$-limit point of $x(t)$.
\end{proposition}
\begin{proof}
Suppose, for contradiction, that $Z\neq \varnothing$ is inclusion-maximal among the zero supports of the $\omega$-limit points of the trajectory. 
 By \Cref{prop:Wprimepos},  there exists  a negative logarithmic limit $a_\infty \in \Delta_Z$ for $(x(t),Z,g)$. 
By \Cref{cor:gooda},  there exists  a face $\Gamma$ of $P_Z$ that is not  $Z$-transverse for $G_Z$,  
$a_\infty \in   N_Z(\Gamma) \cap \Trop^+(S_Z) \cap \Delta_Z$, and $N_Z^*(\Gamma)\cap \R^{Z}_{> 0} \neq \varnothing$. 
 Hence $Z$ is not positively transverse by \Cref{def:siphon}, a contradiction. 
\end{proof}

\begin{theorem}\label{thm:main}
For a positively transverse reaction network,  
no bounded positive trajectory of \eqref{eq:mass} has a boundary $\omega$-limit point. 
\end{theorem}
\begin{proof}
By \cite[Prop. 5.4]{AngeliDeLeenheerSontag2007}, the zero support $Z$ of an $\omega$-limit point is a siphon, and as noted above, $Z$ must be a relevant siphon. The statement now follows 
by applying  \Cref{prop:main} to all inclusion-maximal zero supports of the $\omega$-limit set of a bounded trajectory.
 \end{proof}
 
 We immediately obtain the following corollary on boundary steady states. 
 
 \begin{corollary}\label{cor:steady}
For a positively transverse reaction network with mass-action kinetics,  no positive trajectory converges to a boundary 
steady state. 
\end{corollary}

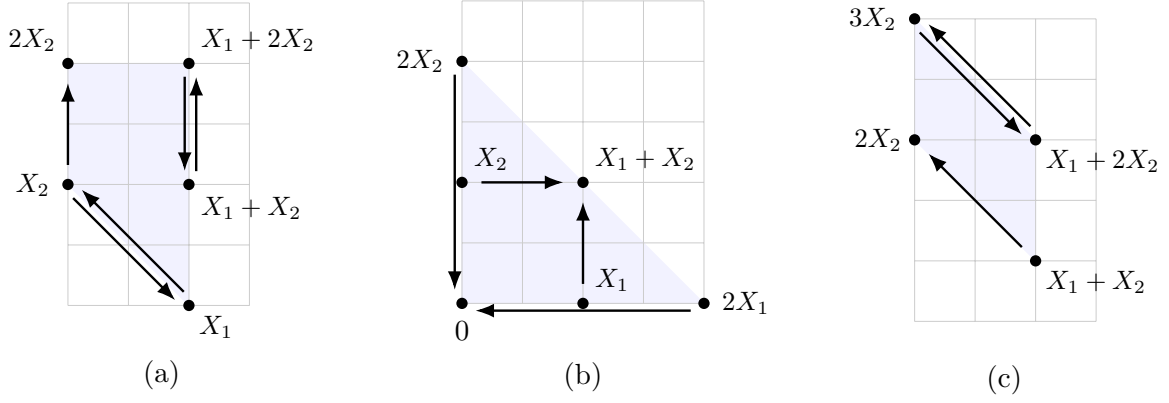
\begin{figure}[t]
\centering
\begin{minipage}[h]{0.3\textwidth}
\centering
\begin{tikzpicture}[scale=0.8,
    x=2cm,y=2cm,
    >=Latex,
    complex/.style={circle, fill=black, inner sep=1.5pt},
    rxn/.style={->, line width=0.9pt, shorten >=5pt, shorten <=5pt},
    gridline/.style={gray!40, line width=0.3pt}
]

\draw[fill=blue!5,draw=blue!5]
    (1,0) --     (0,1) --    (0,2) --    (1,2) --    cycle;
    
\draw[gridline] (0,0) grid (1.5,2.5);

\node[complex] (X1) at (1,0) {};
\node[below right=0pt] at (X1) {\small $X_1$};

\node[complex] (X2) at (0,1) {};
\node[left=3pt] at (X2) {\small  $X_2$};

\node[complex] (2X2) at (0,2) {};
\node[above left=1pt] at (2X2) {\small $2X_2$};

\node[complex] (X1X2) at (1,1) {};
\node[below right=1pt] at (X1X2) {\small $X_1+X_2$};

\node[complex] (X12X2) at (1,2) {};
\node[above right=1pt] at (X12X2) {\small $X_1+2X_2$};

\draw[rxn] (X2) -- (2X2);

\draw[rxn]
  ($(X1)+(0.035,0.035)$) --
  ($(X2)+(0.035,0.035)$);

\draw[rxn]
  ($(X2)+(-0.035,-0.035)$) --
  ($(X1)+(-0.035,-0.035)$);

\draw[rxn]
  ($(X1X2)+(0.06,0)$) --
  ($(X12X2)+(0.06,0)$);

\draw[rxn]
  ($(X12X2)+(-0.035,0)$) --
  ($(X1X2)+(-0.035,0)$);
\end{tikzpicture}
(a)
\end{minipage}\hfill
\begin{minipage}[h]{0.34\textwidth}
\centering
\begin{tikzpicture}[scale=0.8,
    x=2cm,y=2cm,
    >=Latex,
    complex/.style={circle, fill=black, inner sep=1.5pt},
    rxn/.style={->, line width=0.9pt, shorten >=5pt, shorten <=5pt},
    gridline/.style={gray!40, line width=0.3pt}
]

\draw[fill=blue!5,draw=blue!5]
    (0,0) -- (2,0) --      (0,2)  --    cycle;

\draw[gridline] (0,0) grid (2,2.5);

\node[complex] (Z) at (0,0) {};
\node[below=3pt] at (Z) {$0$};

\node[complex] (X1) at (1,0) {};
\node[above right=1pt] at (X1) {\small $X_1$};

\node[complex] (X2) at (0,1) {};
\node[above right=1pt] at (X2) {\small $X_2$};

\node[complex] (2X2) at (0,2) {};
\node[left=3pt] at (2X2) {\small $2X_2$};

\node[complex] (2X1) at (2,0) {};
\node[right=3pt] at (2X1) {\small $2X_1$};

\node[complex] (X1X2) at (1,1) {};
\node[above right=1pt] at (X1X2) {\small $X_1+X_2$};

\draw[rxn]  ($(2X1)+(0,-0.06)$) --   ($(Z)+(0,-0.06)$);
 \draw[rxn]  ($(2X2)+(-0.06,0)$) --   ($(Z)+(-0.06,0)$);
 \draw[rxn] (X1) -- (X1X2);
\draw[rxn] (X2) -- (X1X2);

\end{tikzpicture}
(b)
\end{minipage}\hfill
\begin{minipage}[h]{0.3\textwidth}
\centering
\begin{tikzpicture}[scale=0.8,
    x=2cm,y=2cm,
    >=Latex,
    complex/.style={circle, fill=black, inner sep=1.5pt},
    rxn/.style={->, line width=0.9pt, shorten >=5pt, shorten <=5pt},
    gridline/.style={gray!40, line width=0.3pt}
]

\draw[fill=blue!5,draw=blue!5]
    (1,1) --     (1,2) --    (0,3)   -- (0,2) --    cycle;

\draw[gridline] (0,0.5) grid (1.5,3);

\node[complex] (X1X2) at (1,1) {};
\node[below right=0pt] at (X1X2) {\small $X_1+X_2$};
\node[complex] (X12X2) at (1,2) {};
\node[below right=0pt] at (X12X2) {\small $X_1+2X_2$};
\node[complex] (3X2) at (0,3) {};
\node[left=3pt] at (3X2) {\small $3X_2$};
\node[complex] (2X2) at (0,2) {};
\node[left=1pt] at (2X2) {\small $2X_2$};

\draw[rxn] (X1X2) -- (2X2);
\draw[rxn]   ($(X12X2)+(0.035,0.035)$) --   ($(3X2)+(0.035,0.035)$);
\draw[rxn]   ($(3X2)+(-0.035,-0.035)$) --   ($(X12X2)+(-0.035,-0.035)$);
\end{tikzpicture}

\vspace{0.4cm}

(c)
\end{minipage}
\caption{The three reaction networks discussed in \Cref{ex:ex}.}\label{fig:1}
\end{figure}

 \begin{example}\label{ex:ex}
We consider three networks with two species in \Cref{fig:1}. 
Network (a) is positively transverse. There is only one relevant siphon $Z=\{1,2\}$, which  is positively transverse: the only faces $\Gamma$ for which $N^*_Z(\Gamma)\cap \R^Z_{>0} \neq \varnothing$ are the vertex $X_1$, the vertex $X_2$, or the  edge joining them. Each of them has a reaction with reactant in the face and product outside. Hence the  faces are   $Z$-transverse.

For network (b),  the relevant siphon $Z=\{1,2\}$ has only one face for which  $N^*_Z(\Gamma)\cap \R^Z_{>0} \neq \varnothing$, namely $\Gamma=\{0\}$. There is no reaction with reactant $0$, and hence $\{0\}$ is not  $Z$-transverse.  The third condition in \Cref{def:siphon} also fails, as $S_Z=\R^2$. The network is not positively transverse.

For network (c), the siphon $Z=\{1,2\}$ is not relevant as $(1,1)$ is a nonnegative conservation vector with support in $Z$. The only other siphon $\{2\}$ is positively transverse, and hence so is the network. 
 \end{example}

 \begin{example}\label{ex:trop}
 Consider the following network with three species:
 \begin{align*}
 X_2  & \ce{->} 2X_2+X_3   &     X_1+X_2+2X_3 & \ce{->} X_1 +X_3 \\[5pt]
 X_1+X_2 & \ce{<=>} 2X_1+2X_2 &
  X_1+2X_3   \ce{->} 2X_1+ X_3  & \ce{->} 4X_1+X_2 \,.
 \end{align*}
We show that the network is positively transverse and hence by \Cref{thm:main}, no bounded positive trajectory has a  boundary $\omega$-limit point. 
 The stoichiometric subspace $S$ has dimension $2$ and is generated by $(1,1,0),(0,1,1)$. Hence, 
 \[ S^\perp=\langle (1,-1,1)\rangle\,.\]
From this we obtain that all siphons are relevant. There are four siphons: $\{1\}$, $\{1,2\}$, $\{2,3\}$ and $\{1,2,3\}$. 
 We start by studying $Z\coloneqq \{1,2,3\}$ and consider $a\in \R^3_{>0}$. Since all entries of $a$ are positive, the face  of $P_Z$ determined by $a$ is:
 \[  \begin{cases}
   \conv (X_2,X_1+X_3)   & \text{if } a_2=a_1+a_3  \\
 \{X_2\} & \text{if } a_2<a_1+a_3 \\
  \{X_1+X_3\} & \text{if } a_2>a_1+a_3\,. 
 \end{cases}
 \]
 In the first two cases, the reaction $X_2   \ce{->} 2X_2+X_3$  shows that  the face is   $Z$-transverse. For the face $\Gamma\coloneqq \{X_1+X_3\}$, for which $N_Z(\Gamma)= \{a\in \R^3_{\geq 0} : a_2\geq a_1+a_3\}$, there is no outgoing reaction, so we need to compute the positive tropicalization of $S_Z$. 
As   $S_Z^\perp=\langle (1,-1,1)\rangle$, there are two signed circuits $(C^+,C^-)$ of  $S_Z^\perp$: 
\[ (\{1,3\},\{2\}) \qquad \text{and}\qquad (\{2\},\{1,3\}) \,.\]
 This gives that 
 \[ \Trop^+(S_Z)= \{a\in \R^3 : a_1=a_2\leq a_3 \text{ or }a_2=a_3 \leq a_1 \} = \{a\in \R^3 : a_2=\min \{a_1,a_3\}\}\,. \]
 From this we obtain that if $a\in \Trop^+(S_Z)\cap N_Z(\Gamma)$, then
 \[ a_2\geq a_1+a_3  \geq \min \{a_1,a_3\} = a_2\]
 and hence all inequalities are equalities. 
In particular the equality $a_1+a_3 = \min \{a_1,a_3\}$ with $a_1,a_3\geq 0$ forces $a_1=a_3=0$, hence $a_2=0$, and $a=0\notin \Delta_Z$. We conclude that $\Trop^+(S_Z)\cap N_Z(\Gamma)\cap \Delta_Z=\varnothing$ and hence by \Cref{def:siphon}, $Z$ is positively transverse. 
  The sets $\Trop^+(S_Z)\cap \Delta_Z$ and $N_Z(\Gamma)\cap \Delta_Z$  are illustrated in \Cref{fig:2}.  
 
It is more straightforward to show that the siphons  $\{1\}$, $\{1,2\}$, $\{2,3\}$ are positively transverse, as in this case all faces determined by a positive vector are  $Z$-transverse. 
 \end{example}
 
  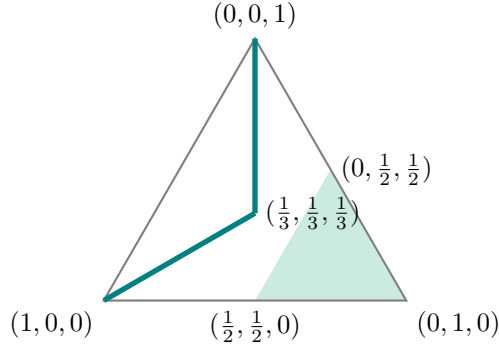
\begin{figure}[t]
\begin{tikzpicture}[scale=1,
    x={(0cm,0cm)},
    y={(4cm,0cm)},
    z={(2cm,3.464cm)}
]

\coordinate (A1) at (1,0,0);
\coordinate (A2) at (0,1,0);
\coordinate (A3) at (0,0,1);

\coordinate (C) at ({1/3},{1/3},{1/3});

\coordinate (P1) at (0,{1/2},{1/2});
\coordinate (P2) at ({1/2},{1/2},0);

\fill[green!60!blue!20!white]
    (P1) -- (P2) -- (A2) -- cycle;

\draw[gray,thick]
    (A1) -- (A2) -- (A3) -- cycle;

\draw[blue!50!green, line width=2pt] (C) -- (A3);
\draw[blue!50!green, line width=2pt] (C) -- (A1);

\node[below left]  at (A1) {\small $(1,0,0)$};
\node[below right] at (A2) {\small  $(0,1,0)$};
\node[above]       at (A3) {\small  $(0,0,1)$};
\node[below]  at ({1/2},{1/2},0) {\small $(\tfrac12,\tfrac12,0)$};
\node[right]  at (0,{1/2},{1/2}) {\small $(0,\tfrac12,\tfrac12)$};
\node[right] at (C) {\small  $(\tfrac13,\tfrac13,\tfrac13)$};

\end{tikzpicture}
 
 \caption{Illustration of  \Cref{ex:trop} for the siphon $Z=\{1,2,3\}$ in the standard simplex $\Delta_Z$. The set $\Trop^+(S_Z)\cap \Delta_Z$ consists of the two segments from $(\tfrac{1}{3},\tfrac{1}{3},\tfrac{1}{3})$ to $(1,0,0)$ and $(0,0,1)$ respectively. The set  $N_Z(\Gamma)\cap \Delta_Z$ is the filled triangle with vertices $(0,1,0),(0,\tfrac{1}{2},\tfrac{1}{2}),(\tfrac{1}{2},\tfrac{1}{2},0)$.    }\label{fig:2}
 \end{figure}

 \begin{remark}
The strategy of proof of \Cref{thm:main} follows the same principles as several earlier works, notably \cite{GMS2,Anderson11}: on the one hand, convergence to the boundary forces some Lyapunov-type function to be decreasing along a sequence of points of the trajectory; on the other hand, the study of the logarithmic limit of these points in conjunction with the geometry of the convex hull of the complexes of the network gives that the function has nonnegative derivative. 
  
 The function $g$ used in this work arises from the evaluation of the derivative of the function $-W$ in the proof of \Cref{prop:Wprimepos}, which was  also used in previous work, e.g., \cite[Thm. 7.5 and \S8.2]{GMS2}. Our chosen evaluation  allows us to more easily control the simultaneous convergence of the coefficients and  the monomials through the normalization of the monomials. 
 This is done at the cost of including complexes that only appear as products of reactions in our analysis of the monomials. 
  With this, we can avoid the tier \cite{Anderson11} and jet   \cite{GMS2} decompositions and provide a direct study of negative logarithmic limits in \Cref{prop:selectpos}.    
  \end{remark}

  \begin{remark}
Every nonrelevant siphon $Z$ satisfies $\Trop^+(S_Z)=\varnothing$ and hence the third condition in \Cref{def:siphon} holds. Requiring all siphons to be positively transverse in \Cref{def:positive} is  thus equivalent to requiring only the relevant siphons to be positively transverse. Checking relevance however is an easy verification worth performing as a first step and the chosen formulation emphasizes this.  
\end{remark}

 \begin{remark}
 In order to  computationally decide whether a network is positively transverse, one can start by identifying all siphons and selecting the relevant ones. Several algorithms in the Petri net literature allow for siphon identification, e.g.,\cite{KinuyamaMurata1986,NabliFagesMartinezSoliman2012}; see also \cite{ShiuSturmfels} for an algebraic approach. After every relevant siphon $Z$ has been identified, the next step is to decide whether $Z$ is positively transverse. To this end, one can start by finding $N_Z^*(\Gamma)$ and $N_Z(\Gamma)$ for all faces of $P_Z$. Among the faces with $N_Z^*(\Gamma)\cap \R^{Z}_{> 0} \neq \varnothing$, find  the set $\mathcal{P}$ of faces that are not  $Z$-transverse by simply identifying the complexes that belong to the face. 

If $\mathcal{P}=\varnothing$, then $Z$ is positively transverse. Otherwise, $Z$ is positively transverse if and only if 
$N_Z(\Gamma)\cap \Trop^+(S_Z)\cap \Delta_Z = \varnothing$ for all $\Gamma\in \mathcal{P}$.  The positive tropicalization of $S_Z$ agrees with the positive Bergman fan of the oriented matroid of $S_Z$. This  can be computed by first computing the Bergman fan of the matroid of $S_Z$ and then following \cite[Alg. 1]{RoseTelek24}. 

All polyhedral computations such as finding inner normal cones and the Bergman fan can be performed, 
for instance, using \texttt{polymake} \cite{polymake} or \texttt{oscar.jl} in Julia \cite{OSCAR,OSCAR-book}. 
 \end{remark}

 \begin{remark}\label{rem:literature}
 In \cite{CNP}   \emph{lower endotactic} networks were introduced and extended in \cite{Pantea} (see also \cite{JohnstonPanteaDonnell2016} for a computational approach to decide lower endotacticity). 
 For every nonnegative vector $a$, consider the subnetwork $G_a=(V_a,E_a)$ of $G$ obtained by removing the reactions $y\rightarrow y'$  satisfying $a\cdot (y'-y)= 0$. 
Then a network is said to be lower endotactic if $a\cdot (y'-y)> 0$ for every reaction $y\rightarrow y'$ of $G_a$ whenever $y\in \argmin \{a \cdot \tilde{y}:  \tilde{y}\rightarrow \tilde{y}'\in E_a\}$.   This is a ``nonnegative version'' of endotactic networks, where this condition is required for all direction vectors. 

Networks (a) and (b) in \Cref{fig:1} are lower endotactic, but (c) is not. Network (c) is in particular not endotactic, and we note also that it is not covered by the siphon-based persistence criterion of \cite{AngeliDeLeenheerSontag2007}, since it has a relevant siphon (called \emph{critical siphon} in loc. cit.), namely $\{2\}$.

Pantea shows that if the network has a two-dimensional stoichiometric subspace and lower endotactic stoichiometric subnetworks, then any bounded positive trajectory  is persistent and in particular does not have boundary $\omega$-limit points \cite[Thm. 5.1]{Pantea}. This was used to show that the GAC holds if the stoichiometric subspace has dimension three.

In \cite{GMS2}, the authors studied  \emph{strongly endotactic networks}, which are endotactic networks such that, for any vector $a\in \R^n$ not in $S^\perp$, there exists a reaction $y\rightarrow y'$ such that 
$a\cdot  y'> a\cdot y$ and $y\in \argmin \{a \cdot \tilde{y}:  \tilde{y}\rightarrow \tilde{y}'\in E\}$.  They proved that strongly endotactic mass-action systems are permanent. In particular their positive trajectories are bounded and persistent, and consequently the GAC holds for strongly endotactic complex-balancing systems.
The approach is revisited in \cite{anderson:cappelletti:tier}, where  consequences of the tier structure of strongly endotactic networks for the stochastic modeling of the network are also obtained.  
 
 Positive transversality may look like a strengthening of lower endotacticity, since we also require a reaction to point into the polytope whenever the face is selected by a positive vector. However, the two classes are incomparable: network (b) in \Cref{fig:1} is lower endotactic but not positively transverse, while network (c)  is  positively transverse but not lower endotactic. 
 \end{remark}

 We conclude by considering  the special case where \eqref{eq:mass} has a positive complex-balancing steady state.  
Then the network is necessarily weakly reversible \cite{Horn1972,Feinberg1972} and \eqref{eq:mass} has bounded trajectories. The existence of one positive complex-balancing steady state forces all positive steady states to be complex-balancing as well and there is exactly one positive steady state in each stoichiometric compatibility class that intersects the positive orthant \cite{HornJackson1972}. 
\Cref{thm:main}   when restricted to systems with positive complex-balancing steady states gives an alternative proof of the   GAC for 
positively transverse networks.  
 
 \begin{theorem}[GAC for positively transverse networks]\label{thm:gac1}
 Consider a   reaction network such that \eqref{eq:mass} admits a positive complex-balancing steady state. 
Assume that  every relevant siphon $Z\neq \{1,\dots,n\}$ that is not a singleton is positively transverse.

Then every positive trajectory $x(t)$  converges to the unique complex-balancing positive steady state  in the stoichiometric compatibility class of $x(0)$.
 \end{theorem}
\begin{proof}
By \cite[Theorem 3.2]{SiegelMacLean}, see also \cite{TDS,Sontag01}, the $\omega$-limit set of a trajectory $x(t)$ with positive initial condition $x(0)$  either  consists of the unique complex-balancing steady state $c^*$  in the stoichiometric compatibility class of $x(0)$ or is contained in the boundary of $\R^n_{\geq 0}$. Since the trajectory is bounded by \cite[Lemma 3.5]{SiegelMacLean}, the trajectory converges to $c^*$ if we can exclude the latter scenario.
  
By \cite[Prop. 20, Rem. 21]{TDS} the origin is not an $\omega$-limit point. 
By \cite[Cor. 3.3]{AndersonShiu10}, a singleton cannot be an inclusion-maximal zero support of an $\omega$-limit point  (as any such point will necessarily lie in the relative interior of a facet of the stoichiometric compatibility class).  
If $x(t)$ has a boundary $\omega$-limit point,  then there exists an inclusion-maximal zero support $Z$ of $\omega$-limit points, which is a relevant siphon, and  is neither a singleton nor $\{1,\dots,n\}$.  By hypothesis, $Z$ is positively transverse, and we obtain a contradiction by \Cref{prop:main}. 
\end{proof}

\small

\begin{thebibliography}{10}

\bibitem{OSCAR}
\normalfont{The OSCAR Team.} \textit{OSCAR -- Open Source Computer Algebra
  Research system}.
\newblock Available at \url{https://www.oscar-system.org}.

\bibitem{alessandrini}
D.~Alessandrini.
\newblock Logarithmic limit sets of real semi-algebraic sets.
\newblock {\em Adv. Geom.}, 13(1):155--190, 2013.

\bibitem{Anderson08}
D.~F. Anderson.
\newblock Global asymptotic stability for a class of nonlinear chemical
  equations.
\newblock {\em SIAM J. Appl. Math.}, 68(5):1464--1476, 2008.

\bibitem{Anderson2011}
D.~F. Anderson.
\newblock Boundedness of trajectories for weakly reversible, single linkage
  class reaction systems.
\newblock {\em J. Math. Chem.}, 49(10):2275--2290, 2011.

\bibitem{Anderson11}
D.~F. Anderson.
\newblock A proof of the global attractor conjecture in the single linkage
  class case.
\newblock {\em SIAM J. Appl. Math.}, 71(4):1487--1508, 2011.

\bibitem{Anderson2020Classes}
D.~F. Anderson, J.~D. Brunner, G.~Craciun, and M.~D. Johnston.
\newblock On classes of reaction networks and their associated polynomial
  dynamical systems.
\newblock {\em J. Math. Chem.}, 58(9):1895--1925, 2020.

\bibitem{anderson:cappelletti:tier}
D.~F. Anderson, D.~Cappelletti, J.~Kim, and T.~D. Nguyen.
\newblock Tier structure of strongly endotactic reaction networks.
\newblock {\em Stoch. Process. Their Appl.}, 130(12):7218--7259, 2020.

\bibitem{AndersonShiu10}
D.~F. Anderson and A.~Shiu.
\newblock The dynamics of weakly reversible population processes near facets.
\newblock {\em SIAM J. Appl. Math.}, 70(6):1840--1858, 2010.

\bibitem{AngeliDeLeenheerSontag2007}
D.~Angeli, P.~De~Leenheer, and E.~D. Sontag.
\newblock A {P}etri net approach to the study of persistence in chemical
  reaction networks.
\newblock {\em Math. Biosci.}, 210(2):598--618, 2007.

\bibitem{ArdilaKlivansWilliams2006}
F.~Ardila, C.~Klivans, and L.~Williams.
\newblock The positive {B}ergman complex of an oriented matroid.
\newblock {\em Eur. J. Comb.}, 27(4):577--591, 2006.

\bibitem{BorosHofbauer2020}
B.~Bal{\'a}zs and H.~Josef.
\newblock Permanence of weakly reversible mass-action systems with a single
  linkage class.
\newblock {\em SIAM J. Appl. Dyn. Syst.}, 19(1):352--365, 2020.

\bibitem{craciun:GAC}
G.~Craciun.
\newblock Toric differential inclusions and a proof of the {G}lobal {A}ttractor
  {C}onjecture.
\newblock Available at {\tt arXiv:1501.02860v3}, 2026.
\newblock Version 3, September 2026.

\bibitem{TDS}
G.~Craciun, A.~Dickenstein, A.~Shiu, and B.~Sturmfels.
\newblock Toric dynamical systems.
\newblock {\em J. Symbolic Comput.}, 44(11):1551--1565, 2009.

\bibitem{CNP}
G.~Craciun, F.~Nazarov, and C.~Pantea.
\newblock Persistence and permanence of mass-action and power-law dynamical
  systems.
\newblock {\em SIAM J. Appl. Math.}, 73(1):305--329, 2013.

\bibitem{OSCAR-book}
W.~Decker, C.~Eder, C.~Fieker, M.~Horn, and M.~Joswig, editors.
\newblock {\em The {C}omputer {A}lgebra {S}ystem {OSCAR}: {A}lgorithms and
  {E}xamples}, volume~32 of {\em Algorithms and {C}omputation in
  {M}athematics}.
\newblock Springer, 1 edition, 2025.

\bibitem{DeshpandeGopalkrishnan2014}
A.~Deshpande and M.~Gopalkrishnan.
\newblock Autocatalysis in reaction networks.
\newblock {\em Bull. Math. Biol.}, 76(10):2570--2595, 2014.

\bibitem{Feinberg1972}
M.~Feinberg.
\newblock Complex balancing in general kinetic systems.
\newblock {\em Arch. Rational Mech. Anal.}, 49:187--194, 1972/73.

\bibitem{Feinberg1987}
M.~Feinberg.
\newblock Chemical reaction network structure and the stability of complex
  isothermal reactors--{I}. {T}he deficiency zero and deficiency one theorems.
\newblock {\em Chem. Eng. Sci.}, 42(10):2229--2268, 1987.

\bibitem{polymake}
E.~Gawrilow and M.~Joswig.
\newblock polymake: a framework for analyzing convex polytopes.
\newblock In {\em Polytopes - combinatorics and computation. DMV-seminar
  Oberwolfach, Germany, November 1997}, pages 43--73. Basel: Birkh{\"a}user,
  2000.

\bibitem{Gopalkrishnan2011Catalysis}
M.~Gopalkrishnan.
\newblock Catalysis in reaction networks.
\newblock {\em Bull. Math. Biol.}, 73(12):2962--2982, 2011.

\bibitem{manoj:cycles}
M.~Gopalkrishnan.
\newblock On the {L}yapunov function for complex-balanced mass-action systems.
\newblock arXiv:1312.3043, 2013.

\bibitem{GMS2}
M.~Gopalkrishnan, E.~Miller, and A.~Shiu.
\newblock A geometric approach to the {G}lobal {A}ttractor {C}onjecture.
\newblock {\em SIAM J. Appl. Dyn. Syst.}, 13(2):758--797, 2014.

\bibitem{Horn1972}
F.~Horn.
\newblock Necessary and sufficient conditions for complex balancing in chemical
  kinetics.
\newblock {\em Arch. Rational Mech. Anal.}, 49:172--186, 1972.

\bibitem{Horn74}
F.~Horn.
\newblock The dynamics of open reaction systems.
\newblock In {\em Mathematical aspects of chemical and biochemical problems and
  quantum chemistry, SIAM--AMS Proceedings, VIII}, pages 125--137, 1974.

\bibitem{HornJackson1972}
F.~Horn and R.~Jackson.
\newblock {General mass action kinetics}.
\newblock {\em Arch. Ration. Mech. Anal.}, 47(2):81--116, 1972.

\bibitem{JohnstonPanteaDonnell2016}
M.~D. Johnston, C.~Pantea, and P.~Donnell.
\newblock A computational approach to persistence, permanence, and
  endotacticity of biochemical reaction systems.
\newblock {\em J. Math. Biol.}, 72:467--498, 2016.

\bibitem{JohnstonSiegel}
M.~D. Johnston and D.~Siegel.
\newblock Weak dynamical nonemptiability and persistence of chemical kinetics
  systems.
\newblock {\em SIAM J. Appl. Math.}, 71(4):1263--1279, 2011.

\bibitem{KinuyamaMurata1986}
M.~Kinuyama and T.~Murata.
\newblock Generating siphons and traps by {P}etri-net representation of logic
  equations.
\newblock In {\em Proceedings of the 2nd IECE (Japan) Conference on Net
  Theory}, pages 93--100, 1986.

\bibitem{Lawrence-thesis}
R.~Lawrence.
\newblock {\em Simplicial Reaction Networks and Dynamics on Graphs}.
\newblock PhD thesis, University of California, Berkeley, 2023.

\bibitem{MarcondesFeliuWiuf2017}
M.~{Marcondes de Freitas}, E.~Feliu, and C.~Wiuf.
\newblock Intermediates, catalysts, persistence, and boundary steady states.
\newblock {\em J. Math. Biol.}, 74(4):887--932, 2017.

\bibitem{NabliFagesMartinezSoliman2012}
F.~Nabli, F.~Fages, T.~Martinez, and S.~Soliman.
\newblock A {B}oolean model for enumerating minimal siphons and traps in
  {P}etri nets.
\newblock In {\em Principles and Practice of Constraint Programming}, volume
  7514 of {\em Lecture Notes in Computer Science}, pages 798--814. Springer,
  2012.

\bibitem{AI-persistence}
{OpenAI}.
\newblock {Boundedness and persistence of weakly reversible mass-action
  systems}.
\newblock OpenAI Math Release preprint
  \href{https://github.com/openai/math/blob/main/preprints/Boundedness-and-persistence-of-weakly-reversible-mass-action-systems-September-25-2026/paper.pdf}{OAI:Boundedness-and-persistence-of-weakly-reversible-mass-action-systems-September-25-2026},
  2026.
\newblock Preprint dated September 25, 2026; publicly released on GitHub on
  October 6, 2026.

\bibitem{AI-permanence}
{OpenAI}.
\newblock {Uniform Permanence in Weakly Reversible Mass-Action Systems}.
\newblock OpenAI Math Release preprint
  \href{https://github.com/openai/math/blob/main/preprints/Uniform-Permanence-in-Weakly-Reversible-Mass-Action-Systems-October-5-2026/permanence.pdf}{OAI:Uniform-Permanence-in-Weakly-Reversible-Mass-Action-Systems-October-5-2026},
  2026.

\bibitem{Pantea}
C.~Pantea.
\newblock On the persistence and global stability of mass-action systems.
\newblock {\em SIAM J.\ Math.\ Anal.}, 44(3):1636--1673, 2012.

\bibitem{RoseTelek24}
K.~Rose and M.~L. Telek.
\newblock Computing positive tropical varieties and lower bounds on the number
  of positive roots.
\newblock {\em J. Symb. Comput.}, 132:102477, 2026.

\bibitem{ShiuSturmfels}
A.~Shiu and B.~Sturmfels.
\newblock Siphons in chemical reaction networks.
\newblock {\em Bull.\ Math.\ Biol.}, 72(6):1448--1463, 2010.

\bibitem{SiegelMacLean}
D.~Siegel and D.~MacLean.
\newblock Global stability of complex balanced mechanisms.
\newblock {\em J. Math. Chem.}, 27(1-2):89--110, 2000.

\bibitem{Sontag01}
E.~D. Sontag.
\newblock {Structure and stability of certain chemical networks and
  applications to the kinetic proofreading model of T-cell receptor signal
  transduction}.
\newblock {\em IEEE Trans. on Automat. Control}, 46(7):1028--1047, 2001.

\bibitem{SpeyerWilliams2005}
D.~Speyer and L.~Williams.
\newblock The tropical totally positive {G}rassmannian.
\newblock {\em J. Algebraic Combin.}, 22(2):189--210, 2005.

\bibitem{volpert}
A.~I. Vol'pert.
\newblock Differential equations on graphs.
\newblock {\em Math. {USSR}-{S}b.}, 17(4):571--582, 1972.

\bibitem{wiuf:GAC}
C.~Wiuf.
\newblock A proof of the {G}lobal {A}ttractor {C}onjecture in a special case.
\newblock arXiv:2609.24553, 2026.

\end{thebibliography}

\end{document}